\documentclass{amsart}
\usepackage{amsfonts,amssymb,amscd,amsmath,enumerate,verbatim,calc}
\usepackage{graphics}

\newcommand{\CM}{Cohen-Macaulay}

\newcommand{\wrt}{with respect to}

\newcommand{\n}{\mathfrak{n} }
\newcommand{\m}{\mathfrak{m} }

\newcommand{\rt}{\rightarrow}

\newcommand{\depth}{\operatorname{depth}}

\theoremstyle{plain}

\newtheorem{thm}{Theorem}

\newtheorem{theorem}{Theorem}[section]
\newtheorem{corollary}[theorem]{Corollary}
\newtheorem{lemma}[theorem]{Lemma}
\newtheorem{proposition}[theorem]{Proposition}

\theoremstyle{definition}
\newtheorem{definition}[theorem]{Definition}
\newtheorem{s}[theorem]{}
\newtheorem{remark}[theorem]{Remark}
\newtheorem{example}[theorem]{Example}

\theoremstyle{remark}

\begin{document}

\title[Relative Hilbert co-efficients ]{ Relative Hilbert co-efficients  }
\author{Amir Mafi, Tony~J.~Puthenpurakal, Rakesh B. T. Reddy and  Hero Saremi}
\date{\today}

\address{Department of Mathematics, University of Kurdistan, P.O. Box: 416, Sanandaj,
Iran.} \email{a\_mafi@ipm.ir}

\address{Department of Mathematics, IIT Bombay, Powai, Mumbai 400 076}
\email{tputhen@math.iitb.ac.in} \email{raki.hcu@gmail.com}

\address{Department of Mathematics, Sanandaj Branch  Islamic Azad University, Sanandaj, Iran.}
\email{hero.saremi@gmail.com}

\subjclass{Primary 13A30; Secondary 13D40, 13D07}
\keywords{associated graded rings, Cohen-Macaulay local rings, reductions, integral closure, normal ideals}
 \begin{abstract}
 Let $(A,\m)$ be a \CM \ local ring of dimension $d$ and let $I \subseteq J$ be two $\m$-primary ideals with $I$ a reduction of $J$. For $i = 0,\ldots,d$ let $e_i^J(A)$ ($e_i^I(A)$) be the $i^{th}$ Hilbert coefficient of $J$ ($I$) respectively. We call the number $c_i(I,J) = e_i^J(A) - e_i^I(A)$ the $i^{th}$ relative Hilbert coefficient of $J$ \wrt \ $I$. If $G_I(A)$ is \CM \ then $c_i(I,J)$ satisfy various constraints. We also show that vanishing of some $c_i(I,J)$ has strong implications on $\depth G_{J^n}(A)$ for $n \gg 0$.
\end{abstract}
 \maketitle
\section*{Introduction}
Let $(A, \m)$ be a Cohen-Macaulay local ring of dimension $d$ and let $J$ be an $\m -$primary ideal. The Hilbert-Samuel 
function of $A$ with respect to $J$ is $H_J(n)=\lambda(A/J^{n+1})$, (here $\lambda(-)$ denotes the length). It is well known that
$H_J$ is of polynomial type $i.e.$ there exists $P_J(X) \in \mathbb{Q}[X]$ such that $H_J(n)=P_J(n)$ for all $n\gg 0$. We write
\[ P_J(X)=e^J_0(A)\binom{X+d}{d}-e^J_1(A)\binom{X+d-1}{d-1}
+\cdots+(-1)^d e^J_d(A).\]
Then the numbers $e^J_i(A)$ for $i=0, 1, \cdots, d$ are the Hilbert coefficients of $A$ with respect to $J$. The number 
$e^J_0(A)$ is called the multiplicity of $A$ with respect to $J$.

Now assume for  convenience $A$ has infinite residue field. Then $J$ has a minimal reduction $\mathfrak{q}$ generated by a system
of parameters of $A$. Let $Gr_J(A)=\bigoplus_{n\geq 0}J^n/J^{n+1}$ be the associated graded ring of $A$ with respect to $J$.
There has been a lot of research regarding properties  of $J$ and $\mathfrak{q}$ and the depth properties of $Gr_J(A)$. For 
example if $J^2=\mathfrak{q}J$ then we say $J$ has minimal multiplicity and in this case $Gr_J(A)$ is Cohen-Macaulay (see, \cite[2.7]{VV}).

In the context of this paper we consider minimal reduction to be an absolute reduction of $J$. The main new idea of this 
paper is that it is convenient to consider reduction $I$ of $J$ not necessarily minimal but having the crucial property that
$Gr_I(A)$ is Cohen-Macaulay. We note that if $\mathfrak{q}$ is a minimal reduction of $J$ then it is generated by system of
parameters of $A$ and so necessarily $Gr_{\mathfrak{q}}(A)$ is Cohen-Macaulay.

As $I$ is a reduction of $J$ then necessarily $e^I_0(A)=e^J_0(A)$. Let \[c_i(I,J)= e^J_i(A)-e^I_i(A) \quad \text{for} \ i\geq 1.\] Then we say
$c_i(I,J)$ to be the $i^{th}$ \textit{relative Hilbert coefficient} of $J$ with respect to $I$. We note that if $\mathfrak{q}$ is a 
minimal reduction of $J$ then $e^{\mathfrak{q}}_i(A)=0$ for $i\geq 1$ and so $c_i(\mathfrak{q},J)=e^J_i(A)$ for $i\geq 1$.

Let us recall the classical Northcott's inequality \[e^J_1(A) \geq e^J_0(A)-\lambda(A/J)\]
(see, \cite{No}). But $e^J_0(A)=\lambda(A/\mathfrak{q})$
where $\mathfrak{q}$ is a minimal reduction of $J$. So Northcott's inequality can be rewritten as 
\[e^J_1(A)\geq \lambda(J/\mathfrak{q}).\]
Furthermore if equality hods then by Huneke (see, \cite{Hminimal}) and Ooishi (see, \cite{Ooishi}) $Gr_J(A)$ is Cohen-Macaulay.

Our first result which easily follows from a deep result of Huckaba and Marley \cite[Theorem 4.7]{HM} is the following.
\begin{thm}\label{thrm1}
 Let $(A,\m)$  be a Cohen-Macaulay local ring and let $I\subset J$ be $\m-$ primary ideals with $I$ a reduction of $J$. Assume
 $Gr_I(A)$ is Cohen-Macaulay. Then
 \[c_1(I, J) \geq \lambda(J/I).\]
 If equality holds then $Gr_J(A)$ is also Cohen-Macaulay.
\end{thm}
We give a different proof of Theorem \ref{thrm1}. Although it's longer than the proof using Huckaba and Marley result, it has
the advantage that it's techniques can be generalized to prove other results.
 
 In \cite{NARITA} Narita proved that $e^J_2(A)\geq 0$. Furthermore if $\dim A= 2$ then $e_2^J(A)=0$ if and only if reduction
 number of $J^n$ is $1$ for $n\gg 0$. In particular $Gr_{J^n}(A)$ is Cohen-Macaulay for $n\gg 0$.
 
 Our genelization of Narita's result is:
 \begin{thm}\label{thrm2}
  Let $(A,\m)$ be a Cohen-Macaulay local ring and let $I\subset J$ be $\m-$ primary ideals with $I$ a reduction of $J$. Assume
 $Gr_I(A)$ is Cohen-Macaulay. Then
 \[c_2(I, J) \geq 0.\]
 If $\dim A =2$ and $c_2(I, J)=0$ then $Gr_{J^n}(A)$ is Cohen-Macaulay for $n\gg 0.$
 \end{thm}
If we assume $J$ is integrally closed then we have the following result:
\begin{thm}\label{thrm3}
 Let $(A,\m)$ be a Cohen-Macaulay local ring of dimension $2$. Let $I\subset J$ be $\m-$ primary ideals with $I$ a reduction
 of $J$. Assume $J$ is integrally closed and  $Gr_I(A)$ is Cohen-Macaulay. If
 $c_1(I, J) = \lambda(J/I) +1$
 then  \[2\lambda(J/I)\leq \lambda(\widetilde{J^2}/I^2) \leq 2 \lambda(J/I)+1.\] If $\lambda(\widetilde{J^2}/I^2) = 2 \lambda(J/I)+1$
 then $Gr_{J^n}(A)$ is Cohen-Macaulay for all $n \gg 0$.
 \end{thm}
 
 In Theorem \ref{thrm3},  $\widetilde{J^2}$ denotes the Ratliff-Rush closure of $J^2$ (see,\cite{RR}).
 
 Narita gave an example which shows that $e_3^J(A)$ of an $\m-$primary ideal $J$ can be negative. Recall that an ideal $J$ is 
 said to be normal if $J^n$ is integrally closed for all $n\geq 1$. In \cite{itoh1} Itoh proved that if $\dim A \geq 3$ and $J$ is a normal ideal
 then $e^J_3(A)\geq 0$. We prove:
 \begin{thm}
  Let $(A,\m)$ be a Cohen-Macaulay local ring of dimension $d \geq 3$. Let $I\subset J$ be $\m-$ primary ideals with $I$ a reduction
 of $J$. Assume $J$ is normal and  $Gr_I(A)$ is Cohen-Macaulay. Then
 \[c_3(I, J)\geq 0. \]
 If $d=3$ and $c_3(I, J)=0$ then $Gr_{J^n}(A)$ is Cohen-Macaulay for all $n \gg 0.$
 \end{thm}

In the main body of the paper we consider a more general situation $(A,\m) \rt (B,\n)$ is a finite map with $\dim A = \dim B$, $I$ an $\m$-primary $A$-ideal, $J$ an $\n$-primary ideal with $IB$ a reduction of $J$. We now describe in brief the contents of this paper. In section one we discuss a few prelimary results that we need. In section two we prove Theorem 1. In section three we prove Theorems 2,3. We prove Theorem 4 in section 4. Finally in section 5 we give a few examples which illustrates our results.

\section{Preliminaries}
Throughout this paper we follow the following hypothesis unless otherwise stated.\\
\textbf{Hypothesis}:  Let $(A, \m)$ $\stackrel{\psi}\longrightarrow$ $(B, \n)$ be a local homomorphism of Cohen-Macaulay local rings such that\\
$(1)$ $B$ is  finite as an $A-$module and  $\dim A = \dim B$.\\
$(2)$ $I\subset A$  and 
 $J \subset B$ are ideals such that $\psi(I)B$ is a reduction of $J$.  
 
 \begin{remark}
 $\psi(I)B$ is not necessarily a minimal reduction of $J$.
 \end{remark}
 
\begin{remark}
 Note that $B/\mathfrak{n}$ is a finite extension of $A/\mathfrak{m}.$ Set $\delta = \dim_{A/\mathfrak{m}}B/\mathfrak{n}. $ Then
for any $B-$module $N$ of finite length we have $\lambda_A(N)=\delta \lambda_B(N).$
\end{remark}
 The following result gives a necessary and sufficient condition for $\psi(I)B$ to be a reduction of $J$.
 \begin{lemma}\label{red-finite-Gr}
Let $\psi: (A, \m) \longrightarrow (B, \mathfrak{n})$ be a local homomorphism of Cohen-Macaulay local rings such that
$B$ is a finite $A-$module and $\dim A = \dim B$. Let $I$ ideal in $A$ and let $J$ ideal in $B$ with $\psi(I)\subset J$. Then $Gr_J(B)$ is 
 finitely generated as a $Gr_I(A)-$module if and only if  $\psi(I)B$ is a reduction of $J$.
\end{lemma}
\begin{proof}
 Suppose $\psi(I)B$ is a reduction of $J$. Let $c \geq 1$ be such that $\psi(I)J^n = J^{n+1}$ for all $n \geq c$. Then 
 $J^{n+1}/J^{n+2}= \psi(I)(J^n/J^{n+1})$ for all $n \geq c$. Therefore $Gr_J(B)$ is a finite $Gr_I(A)-$module.
 
 Conversely suppose that $Gr_J(B)$ is a finite $Gr_I(A)-$module. Then there exists $n_0$ such that $Gr_I(A)_1Gr_J(B)_n=Gr_J(B)_{n+1}$ for all $n \geq n_0.$ Thus for $n \geq n_0$ 
 \begin{align*}
  \frac{J^{n+1}}{J^{n+2}} & = \frac{I}{I^2}.\frac{J^n}{J^{n+1}}\\
  & = \frac{\psi(I)J^n + J^{n+2}}{J^{n+2}}.
 \end{align*}
 So  $J^{n+1}= \psi(I)J^{n}+ J^{n+2}.$ Therefore by Nakayama's lemma $J^{n+1}= \psi(I)J^{n}$ for all $n \geq n_0$.
\end{proof}


\begin{remark}
 Let $\mathcal{R}(I, A)= A[It]= \bigoplus_{n\geq 0}I^nt^n$ be the Rees ring of $A$ with respect to $I.$ If $M$ is a finite $A$ module then
 set $\mathcal{R}(I, M) = \bigoplus_{n \geq 0}I^nMt^n$ the Rees module of $M$ with respect to $I.$ It can also be easily shown $\psi(I)B $ 
 is a reduction of $J$ if and only if $\mathcal{R}(J, B)$ is a finite $\mathcal{R}(I, A)-$module.
\end{remark}
\begin{remark}
 Set $W=\bigoplus_{n\geq 0}J^{n+1}/I^{n+1}B$. Then we have 
 \[0\longrightarrow \mathcal{R}(I, B)\longrightarrow \mathcal{R}(J, B)\longrightarrow W(-1)\longrightarrow 0\]
an exact sequence of $\mathcal{R}(I, A)$ modules . So $W(-1)$ is a finite  $\mathcal{R}(I, A)$ module. Hence $W$ is 
a finite $\mathcal{R}(I, A)-$module.
\end{remark}
The following is our main object to study associated graded modules and Hilbert coefficients.
\begin{definition}
 Let $M$ be an $A-$module. Set $L^I(M)=\bigoplus_{n \geq 0}M/I^{n+1}M.$ Then the  $A-$module $L^I(M)$ can be given an $\mathcal{R}(I,A)-$module 
 structure as follows. The Rees ring $\mathcal{R}(I,A)$ is a subring of $A[t]$ and so $A[t]$ is an $\mathcal{R}(I,A)-$module. Therefore
 $M[t] = M \otimes_A A[t]$ is an $\mathcal{R}(I,A)-$module. The exact sequence 
 \[0\longrightarrow  \mathcal{R}(I,M) \longrightarrow M[t] \longrightarrow L^I(M)(-1) \longrightarrow 0\]
 defines an $\mathcal{R}(I,A)-$module structure on $L^I(M)(-1)$ and so on $L^I(M).$
 Notice $L^I(M)$ is not a finitely generated $\mathcal{R}(I,A)$-module.
\end{definition}

\begin{remark}\label{L^I-mod-u}
 Let $x$ be $M$ superficial with respect to $I$ and set $u=xt\in \mathcal{R}(I, A)_1$. Notice that $L^I(M)/uL^I(M) = L^I(M/xM).$
\end{remark}
 By \cite[Proposition  5.2]{TJ1} we have the following:
\begin{remark}\label{x^*-u=xt-reg}
 Let $x \in I\backslash I^2.$ Then $x^* \in Gr_I(A)_1$ is $Gr_I(M)-$regular if and only if $u = xt \in \mathcal{R}(I, A)_1$ is $L^I(M)-$regular.
\end{remark}

\s Let $\psi:A\longrightarrow B$ as before and $\dim A =d$. Assume that $I$ is $\m-$primary and $J$ is $\mathfrak{n}-$primary.
Define $L^I(B) = \bigoplus_{n \geq 0}B/I^{n+1}B$ and $L^J(B) = \bigoplus_{n \geq 0}B/J^{n+1}.$ As $L^J(B)$ is a $\mathcal{R}(J,B)-$
module and so as a $\mathcal{R}(I,A)-$module. For each $n\geq 0$ we have 
\[0\longrightarrow \frac{J^{n+1}}{I^{n+1}B}\longrightarrow \frac{B}{I^{n+1}B}\longrightarrow \frac{B}{J^{n+1}}\longrightarrow 0\]
an exact sequence of $A-$ modules.  So we get
\[0\longrightarrow W \longrightarrow L^I(B)\longrightarrow L^J(B)\longrightarrow 0\]
an exact sequence of $\mathcal{R}(I, A)$ modules. Therefore 
\[\sum \lambda_A \left(\frac{J^{n+1}}{I^{n+1}B}\right)z^n = \frac{h^I_B(z)- \delta h^J_B(z)}{(1-z)^{d+1}}. \]
Note that  $h^I_B(1)- \delta h^J_B(1)=0.$ So we can write  
\[ \delta h^J_B(z) = h^I_B(z) + (z-1)r(z).\]
Therefore 
we have\\
$1)$ $\delta e^J_0(B)  = e^I_0(B)$.\\
$2)$  $\delta e^J_i(B) = e^I_i(B) + (r^{(i-1)}(1)/(i-1)!)$ for $i\geq 1.$
 
\begin{remark}
 If $\delta e^J_1(B) \not = e^I_1(B)$ then $\dim W =d.$
\end{remark}

 We need the following technical result.
 \begin{lemma}\label{technical}
  Let $\psi: A \longrightarrow B$ as before. Assume the residue field of $A$ is infinite. Then there exists $x \in I$ such that\\
  $(1)$ $x$ is $A$ superficial with respect to $I.$\\
  $(2)$ $\psi(x)$ is $B$ superficial with respect to $J.$
 \end{lemma}
 \begin{proof}
  Note that $Gr_J(B)$ is a finite $Gr_I(A)-$module. Also $\psi$ induces a natural map $\hat{\psi}: Gr_I(A)\longrightarrow Gr_J(B).$ Let 
  $z \in Gr_I(A)_1$ be a $Gr_I(A)\oplus Gr_J(B)$ filter regular. Then note that $\hat{\psi}(z)$ is $Gr_J(B)$ filter regular. Let $x \in I$
  be such that $x^* = z.$ Then clearly $x$ is $A$ superficial with respect to $I$. Also note that $\psi(x)^* = \hat{\psi}(z)$. So $\psi(x)$
  is $B$ superficial with respect to $J$.
 \end{proof}
The following result easily follows by induction on the dimension of the ring.
\begin{corollary}\label{cor-technical-lem}
 Let $\psi : A \longrightarrow B$ as before. Assume that the residue field of $A$ is infinite. Let $\dim A = d \geq 1$. Then there exist 
 $\underline{x} = x_1, \cdots , x_d \in I$ such that\\
 $(1)$ $\underline{x}$ is $A$ superficial sequence with respect to $I$.\\
 $(2)$ $\psi(\underline{x})$ is $B$ superficial sequence with respect to $J$.
\end{corollary}
\begin{proof}
 Follows easily by induction on $d$ and using Lemma $\ref{technical}$.
\end{proof}
If $\psi(I)B\subsetneq J$ then $W \not = 0$. More over we have the following result:
\begin{lemma}
 Let $\psi : A \longrightarrow B$ as before.   Assume that $Gr_I(B)$ is Cohen-Macaulay. Then the following are equivalent:\\
 $(1)$ $\delta e_1^J(B) = e_1^I(B)$.\\
 $(2)$ $Gr_I(B) = Gr_J(B)$. 
\end{lemma}
\begin{proof}
 We first prove that $Gr_J(B)$ is Cohen-Macaulay. By Sally descent and Lemma \ref{technical} we may assume that $\dim B =1$. Now consider the
 exact sequence
 \[0\longrightarrow W \longrightarrow L^I(B)\longrightarrow L^J(B) \longrightarrow 0,\]
 where $W=\bigoplus J^{n+1}/I^{n+1}B$. As $\delta e_1^J(B) = e_1^I(B)$ we get $\lambda(W) < \infty $. Let $\mathfrak{M}$ be the unique homogeneous
 maximal ideal of $\mathcal{R}(I, A)$ and $H^i(-)=H^i_{\mathfrak{M}}(-)$. As $Gr_I(B)$ is Cohen-Macaulay by Remark \ref{x^*-u=xt-reg} 
 we get $H^0(L^I(B))=0$. So we get $W =0$. Therefore $L^I(B) = L^J(B)$.
 So $H^0(L^J(B)) = 0$. Thus $Gr_J(B)$ is Cohen-Macaulay.
 
 Now assume $\dim B \geq 2$. We prove $L^I(B) = L^J(B)$. This will prove the result. Note that $\depth W \geq 1$. Set $u : = xt$ where $x$
 is $A-$superficial with respect to $I$ and $\psi(x)$ is $B-$superficial with respect to $J$. Then we have an exact sequence
 \[0\longrightarrow \frac{W}{uW} \longrightarrow \frac{L^I(B)}{uL^I(B)} \longrightarrow \frac{L^J(B)}{uL^J(B)} \longrightarrow 0.\]
 By induction and Remark \ref{L^I-mod-u} we get $L^I(B)/uL^I(B) = L^J(B)/uL^J(B)$. So $W = uW$. By graded Nakayama's Lemma $W = 0$. Hence $L^I(B) = L^J(B)$.
\end{proof}

\section{Extension of Northcott's inequality.}

The following result easily follows from a result due to Huckaba and Marley  (see, \cite[Theorem 4.7]{HM}). However our techniques to prove
this extends to prove our
other results.
\begin{theorem}\label{northcott-extn}
 Let $\psi : A \longrightarrow B$ as before.
 Assume that $I$ is $\m-$primary and $J$ is $\mathfrak{n}-$primary. If $Gr_I(B) $ is  Cohen-Macaulay then $\delta e^J_1(B) \geq e^I_1(B)+ \lambda(J/IB).$ 
  If  equality holds then $Gr_J(B)$ is also Cohen-Macaulay.
\end{theorem}
\begin{proof}
  By Sally machine we may assume that $\dim A =1.$ Set $W_i=J^{i+1}/I^{i+1}B$ and $W= \bigoplus_{i\geq 0} W_i$. Note that $\lambda(W_i) = e_0(W)$ for $i\gg 0.$
   Let $(x)$ be a minimal reduction of $I.$ Let
 $u=xt \in \mathcal{R}(I, A)_1$.  Then we have 
 \begin{center}
$
\begin{CD}
0 @>>> W_i @>>> L^I(B)_i @>>> L^J(B)_i @>>> 0 \\
  @.       @VV\text{u}V @VV\text{u}V @VV\text{u}V\\
0 @>>> W_{i+1} @>>> L^I(B)_{i+1} @>>> L^J(B)_{i+1} @>>> 0 
\end{CD}
$
\end{center}
a commutative diagram with exact rows. As $Gr_I(B)$ is Cohen-Macaulay so $x^*$ is $Gr_I(B)$ regular. So by Remark \ref{x^*-u=xt-reg}, $u$ is $L^I(B)$ regular. Thus $L^I(B)_i\stackrel{u}\longrightarrow L^I(B)_{i+1}$ 
is injective. As $W \subset L^I(B)$ we also get $W_i \stackrel{u} \longrightarrow W_{i+1}$ is injective. Therefore 
$\lambda(W_0)\leq \lambda(W_1)\leq \cdots \leq \lambda(W_i) = e_0(W) $ for $i \gg 0.$ So
 \begin{align*}
\delta e_1^J(B) & = e_1^I(B) + e_0(W)\\
 & \geq e_1^I(B) + \lambda(J/IB).
 \end{align*}
 Now suppose $\delta e_1^J(B) = e_1^I(B) + \lambda(J/IB).$ 
 Then $e_0(W) = \lambda(J/IB).$
Thus $\lambda(W_n)=\lambda(W_0)$ for all $n\geq 0.$ So we get $W_i\stackrel{u}\longrightarrow W_{i+1}$ is an isomorphism. By 
Snake lemma we get
\[L^J(B)_i \stackrel{u}\longrightarrow L^J(B)_{i+1}\]
is injective. Therefore $u$ is $L^J(B)$ regular. By Remark \ref{x^*-u=xt-reg} $x^*$ is $Gr_J(B)-$regular.
  Hence $Gr_J(B)$ is Cohen-Macaulay.
\end{proof}
Now we give an example where Theorem \ref{northcott-extn} holds.
\begin{example}
 Let $A = \mathbb{Q}[|X,Y,Z,W|]/(XY-YZ, XZ+Y^3-Z^2)$. Let $x,y,z,w$ be the images of $X,Y,Z,W$ in $A$ respectively. Set $\mathfrak{m}=(x,y,z,w)$.
 Then $(A, \mathfrak{m})$ is a two dimensional  Cohen-Macaulay local ring. Let $I=(x,y,w)$. Note that $I$ is $\mathfrak{m}-$primary and $z$
 is integral over $I$. So $I$ is a reduction of $\mathfrak{m}$. It is proved in \cite[example 3.6]{JR} $Gr_I(A)$ is Cohen-Macaulay. 
 Using CoCoA (see, \cite{cocoa}) we have  computed $e_1^I (A) = 6$, $e^{\m}_1(A)=7$ and 
 $\lambda(A/I)=2$. So $e^{\m}_1(A)=e_1^I(A) + \lambda(\m/I)$. Hence by Theorem \ref{northcott-extn} we get $Gr_{\m}(A)$ is Cohen-Macaulay.
\end{example}

The following example shows that the condition $Gr_I(B)$ is Cohen-Macaulay is essential.

\begin{example}\label{CM-essential}
 Let $A= \mathbb{Q}[|X,Y,Z,U,V,W|]/(Z^2,ZU,ZV,UV,YZ-U^3,XZ-V^3)$, with $X,Y,Z,U,V, W$ inderterminates. Let $x,y,z,u,v, w$ be the images of $X,Y,Z,U,V, W$
 in $A$. Set $\mathfrak{m}=(x,y,z,u,v,w)$. Then $(A, \mathfrak{m})$ is a three dimensional Cohen-Macaulay local ring. Let $I=(x,y,u,w)$.
 Note that $v^4= vxz=0$ and $z^2 =0$ in $A$.  Thus $z, v$ are integral over $I$. So $I$ is a reduction of $\mathfrak{m}$. Let $J=(x,y,w)$. Then $J$
 is a minimal reduction $I$. Using CoCoA (see, \cite{cocoa}) we have checked that  
 \[P_I(t)= \frac{4+t^2+t^3}{(1-t)^3} \quad \text{and} \quad P_{\mathfrak{\m}}(t) =\frac{1+3t+3t^3-t^4}{(1-t)^3}.\]
 We also checked $\lambda(I/J)=2, \lambda(I^2/I^2 \cap J)=1$ and $\lambda(I^3/I^3 \cap J)=0.$ Therefore by
 \cite[Theorem 4.7]{HM} we get
 $\depth Gr_I(A)< 3.$ Hence $Gr_I(A)$ is not Cohen-Macaulay. Also note that $h-$polynomial of $Gr_{\m}(A)$ has negative coefficient. So
 $Gr_{\m}(A)$ is also not Cohen-Macaulay. It is easy to see that  $e_1^{\m}(A)=e_1^I(A) + \lambda(\m/I).$
\end{example}

\section{The case of dimension two}
 Let $\mathfrak{a}$ be an ideal in a Notherian ring $S$ and $M$ a finite $S$ module. Then for $n \geq 1$,
$\widetilde{\mathfrak{a}^nM} := \cup_{k\geq 1}(\mathfrak{a}^{n+k}M:_M\mathfrak{a}^k)$
is called the {\it Ratliff-Rush} closure of $\mathfrak{a}^nM$.

In general if $\mathfrak{a}\subset \mathfrak{b}$ be two ideals in a ring $S$ then it need not imply that
$\widetilde{\mathfrak{a}} \subset \widetilde{\mathfrak{b}}$. However for 
reduction of ideals we have the following:
\begin{proposition}
 Let $S$ be a Notherian ring and let $\mathfrak{a} \subset \mathfrak{b}$ be a reduction of  $\mathfrak{b}$. 
 Then $\widetilde{\mathfrak{a}^{n}} \subset \widetilde{\mathfrak{b}^{n}} $ for all $n \geq 1$.
\end{proposition}
\begin{proof}
 Let $x \in \widetilde{\mathfrak{a}^n}$. So $x\mathfrak{a}^k \subset \mathfrak{a}^{n+k}$ for some $k.$ Thus 
 $x\mathfrak{a}^k\mathfrak{b}^r \subset \mathfrak{a}^{n+k}\mathfrak{b}^r$ for all $r\geq 0.$
As $\mathfrak{a} \subset \mathfrak{b}$ is a reduction so $\mathfrak{a}\mathfrak{b}^s =\mathfrak{b}^{s+1}$ for $s \gg 0.$ Choose $r \geq s.$
Then $\mathfrak{a}^k\mathfrak{b}^r =\mathfrak{b}^{k+r}.$ Therefore $x\mathfrak{b}^{k+r} \subset \mathfrak{b}^{n+k+r}.$ Thus $x \in \widetilde{\mathfrak{b}^n}. $
 \end{proof}

\s Let $M$ be an $A-$module. Define $\widetilde{L}^I(M)= \bigoplus_{n \geq 0}M/\widetilde{I^{n+1}M}.$ Then $\widetilde{L}^I(M)$ is a
$\mathcal{R}(\widetilde{I}, A)-$module so $\mathcal{R}(I, A)-$module. Set 
\[\widetilde{L}^J(B) = \bigoplus_{n \geq 0}B/\widetilde{J^{n+1}} \quad \text{and} \quad   \widetilde{W}= \bigoplus_{n \geq 0}\widetilde{J^{n+1}}/I^{n+1}B.\]
Then we have 
\[0\longrightarrow \widetilde{W} \longrightarrow L^I(B) \longrightarrow \widetilde{L}^J(B) \longrightarrow 0 \]
 an exact sequence of $\mathcal{R}(I, A)$ modules. Note that $h_B^I(1) = \delta h_B^J(1) = \delta \widetilde{h}_B^{J}(1).$
 Therefore we can write 
 \[\delta \widetilde{h}_B^{J}(z) = h_B^I(z) + (z-1)\widetilde{r}(z) \quad \text{and} \quad H_{\widetilde{W}}(z) = \frac{\widetilde{r}(z)}{(1-z)^d}.\]
 Therefore\\
 $1)$ $\delta \widetilde{e}_0^{J}(B)= e_0^I(B).$\\
 $2)$ $\delta \widetilde{e}_i^{J}(B)= e_i^I(B) + \widetilde{r}^{(i-1)}(1)/(i-1)!.$
 
 Now we extend a famous result of Narita concerning second Hilbert coefficient (see, \cite{NARITA}).
 \begin{proposition}\label{W-hat-CM}
  Let $\psi : A\longrightarrow B$ as before and $\dim A \geq 2.$ Let $I$ be $\m-$primary and $J$ be $\mathfrak{n}-$primary.
  Assume that $Gr_I(B)$ is Cohen-Macaulay. Then 
  \[\delta e_2^J(B) \geq e_2^I(B).\] 
 \end{proposition}
 \begin{proof}
 We may assume that $\dim A =2$. Let $\mathfrak{M}$ be the unique homogeneous maximal ideal of $\mathcal{R}(I, A)$. Let $H^i (-) := H_{\mathfrak{M}}(-)$ be the $i^{th}$ 
  local cohomology  module. As $Gr_I(B)$ is Cohen-Macaulay so $H^i(Gr_I(B)) =0$ for $i = 0, 1.$ Also note that 
  $H^0(\widetilde{Gr}_J(B))=0.$  By Remark \ref{x^*-u=xt-reg} 
 $H^0(\widetilde{L}^J(B)) =0$ and  $H^i(L^I(B)) =0$ for $i = 0, 1.$
  As we have 
\[0\longrightarrow \widetilde{W} \longrightarrow L^I(B) \longrightarrow \widetilde{L}^J(B) \longrightarrow 0 \]
 an exact sequence of $\mathcal{R}(I, A)$ modules. By considering long exact sequence in local cohomology we get $H^i(\widetilde{W})=0$ for 
 $i = 0 , 1.$ Hence $\widetilde{W}$ is Cohen-Macaulay. So $\widetilde{r}^{(1)}(1) \geq 0.$ Also note that $e_2^I(B) \geq 0$.
 
 Now \begin{align*}
 \delta \widetilde{e}_2^{J}(B) & = e_2^I(B) + \widetilde{r}^{(1)}(1).\\
     & \geq  e_2^I(B).
    \end{align*}
   \end{proof}
   
    \begin{remark}\label{rem-W-hat-CM}
  From the proof of Proposition \ref{W-hat-CM} one can see that $\widetilde{W}$ is Cohen-Macaulay if $\dim A =2$.
 \end{remark}

By analysing the case of equality in the above Theorem we prove:
 
\begin{theorem}\label{ext-narita}
Let $\psi : A\longrightarrow B$ as before and $\dim A =2.$ Let $I$ be $\m-$primary and $J$ be $\mathfrak{n}-$primary.  Assume that 
 $Gr_I(B)$ is Cohen-Macaulay. Suppose $\delta e_2^J(B) = e_2^I(B).$ Then $\widetilde{Gr}_J(B)$ is Cohen-Macaulay. Consequently $Gr_{J^n}(B) $
 is Cohen-Macaula y for $n \gg 0.$
\end{theorem}
\begin{proof}
 We have  $\delta e_2^J(B) = e_2^I(B) + \widetilde{r}^{(1)}(1).$ By remark \ref{rem-W-hat-CM} we get  $\widetilde{W}$ is Cohen-Macaulay.
 So $\widetilde{r}^{(1)}(1)\geq 0.$ By hypothesis $\delta e_2^J(B) = e_2^I(B).$ So $\widetilde{r}^{(1)}(1) =0.$ Therefore $\widetilde{r}(z) = c$(constant)
 and hence 
 \[H_{\widetilde{W}}(z) = \frac{c}{(1-z)^2}.\]
 Let $x, y$ be a $I$ superficial sequence. Set $u=xt, v=yt.$ Then $ u ,v \in \mathcal{R}(I, A)_1 .$ Now set $\overline{\widetilde{W}}= \widetilde{W}/u\widetilde{W},$ $ \overline{L^I(B)} =
 L^I(B)/uL^I(B)$ and $\overline{\widetilde{L}^J(B)}= \widetilde{L}^J(B)/u\widetilde{L}^J(B).$ As $u$ is $\widetilde{W}\oplus L^I(B)\oplus \widetilde{L}^J(B)-$regular
 we get 
 \[0\longrightarrow \overline{\widetilde{W}} \longrightarrow \overline{L^I(B)} \longrightarrow \overline{\widetilde{L}^J(B)} \longrightarrow 0 \]
 an exact sequence. So we get a commutative diagram
  \begin{center}
$
\begin{CD}
0 @>>> \overline{\widetilde{W}_i} @>>> \overline{ L^I(B)_i} @>>> \overline{\widetilde{L}^J(B)_i} @>>> 0 \\
  @.       @VV\text{v}V @VV\text{v}V @VV\text{v}V\\
0 @>>> \overline{\widetilde{W}_{i+1}} @>>> \overline{ L^I(B)_{i+1}} @>>> \overline{\widetilde{L}^J(B)_{i+1}} @>>> 0 
\end{CD}
$
\end{center}
with exact rows. As $Gr_I(B)$ and $\widetilde{W}$ are Cohen-Macaulay we get that $v$ is $\overline{L^I(B)}\oplus \overline{\widetilde{W}}-$regular.
So $\overline{L^I(B)_i}\stackrel{v}\longrightarrow \overline{L^I(B)_{i+1}}$ and $\overline{\widetilde{W}_i}\stackrel{v} \longrightarrow
\overline{\widetilde{W}_{i+1}}$ are injective. As
$\lambda(\overline{\widetilde{W}_i}) = \lambda(\overline{\widetilde{W}_{i+1}})$ for all $i$. We get $\overline{\widetilde{W}_i}\stackrel{v} \longrightarrow \overline{\widetilde{W}_{i+1}}$ is an isomorphism. By 
Snake Lemma it follows
that $\overline{\widetilde{L}^J(B)_i}\stackrel{v} \longrightarrow \overline{\widetilde{L}^J(B)_{i+1}}$ is injective. 
So $\depth \widetilde{L}^J(B) \geq 2.$ Hence by remark \ref{x^*-u=xt-reg} $\depth \widetilde{Gr}_J(B) \geq 2.$ So $\widetilde{Gr}_J(B)$ is Cohen-Macaulay.
In particular $Gr_{J^n}(B)$ is Cohen-Macaulay for $n \gg 0.$
\end{proof}
Now we give an example where Theorem \ref{ext-narita} holds.
\begin{example}
 Let $A= \mathbb{Q}[|X,Y,Z,U,V|]/(Z^2,ZU,ZV,UV,Y^2Z-U^3,X^2Z-V^3)$, with $X,Y,Z,U,V$ inderterminates. Let $x,y,z,u,v$ be the images of $X,Y,Z,U,V$
 in $A$. Set $\mathfrak{m}=(x,y,z,u,v)$. Then $(A, \mathfrak{m})$ is a two dimensional Cohen-Macaulay local ring. Let $I=(x,y,z,u)$ and 
 $J=(x,y,z,u,v^2)$. Note that $v^4 - vx^2z=0$ in $A$. Thus $v$ is integral over $I$. So $I$ is a reduction of $J$. Let $\mathfrak{q}=(x,y)$. Then $\mathfrak{q}$
 is a minimal reduction of $I$. Using CoCoA (see, \cite{cocoa}), we have computed $e_1^I(A)=4$ and  $e_2^I(A) = e_2^{J}(A)=1$. We also checked that 
 $\lambda(I/\mathfrak{q})=3$, $\lambda(I^2/I^2 \cap \mathfrak{q})=1$. By \cite[Theorem 4.7]{HM} we get $Gr_I(A)$ is Cohen-Macaulay.
 Hence by Theorem \ref{ext-narita} we get $\widetilde{Gr}_J(A)$ Cohen-Macaulay. Hence $Gr_{J^n}(A)$ is Cohen-Macaulay for $n \gg 0$.
\end{example}

For integrally closed ideals we prove:
\begin{theorem}\label{extn-sally-thrm}
Let $\psi : A\longrightarrow B$ as before and $\dim A =2.$ Let $I$ be $\m-$primary and $J$ be $\mathfrak{n}-$primary.
Assume that  $Gr_I(B)$ is Cohen-Macaulay. Suppose $J$ is integrally closed and 
 \[\delta e_1^J(B) = e_1^I(B) + \lambda(J/IB) + 1.\]
 Then\\
  $a)$ $2\lambda(J/IB) \leq \lambda(\widetilde{J^2}/I^2B) \leq 2 \lambda(J/IB) +1.$\\
 $b)$ If $\lambda(\widetilde{J^2}/I^2B) = 2  \lambda(J/IB) +1.$ Then $\widetilde{Gr}_{J}(B)$ is Cohen-Macaulay. Consequently 
 $Gr_{J^n}(B)$ is Cohen-Macaulay for $n \gg 0$.  
\end{theorem}
\begin{proof}
  By remark \ref{rem-W-hat-CM}  $\widetilde{W}$ is Cohen-Macaulay. Let
 \[H_{\widetilde{W}}(z) = \sum_{n\geq 0} \lambda_A({\widetilde{W}}_n)z^n = \frac{\widetilde{r}(z)}{(1-z)^2}  \]
 be the Hilbert series of $\widetilde{W}.$ Note that all the co-efficients of $\widetilde{r}(z)$ are positive. Write
 \[\widetilde{r}(z)= r_0 + r_1z + \cdots + r_sz^s.\]
 Then we have $r_0 = \lambda(\widetilde{J}/IB) = \lambda(J/IB) $ and $\lambda(\widetilde{J^2}/I^2B) = 2r_0 + r_1.$ We have
 \begin{align*}
  \delta e_1^J(B) & = e_1^I(B) + \widetilde{r}(1).\\
  & = e_1^I(B) + \lambda(J/IB) + 1. 
 \end{align*}
Therefore $\widetilde{r}(1) = \lambda(J/IB) + 1.$ Hence $r_1 + \cdots + r_s = 1. $ So $(a)$ follows.

Suppose $(b)$ holds. $i.e.$
\[\lambda(\widetilde{J^2}/I^2B) = 2  \lambda(J/IB) +1.\]
Then $r_1 = 1$ and $r_j =0$ for $j \geq 2.$ Let $x, y$ be a $I$ superficial sequence. Set $u =xt.$ Then $u \in \mathcal{R}(I, A)_1.$
Also set $U=\widetilde{W}/u\widetilde{W}.$ Then we have 
\[0\longrightarrow U \longrightarrow \frac{L^I(B)}{uL^I(B)}\longrightarrow \frac{\widetilde{L}^J(B)}{u\widetilde{L}^J(B)} \longrightarrow 0\]
an exact sequence. Note that $\lambda(U_n) = \lambda(U_2)$ for $n \geq 2.$ Also note that $v = yt $ in $\mathcal{R}(I, A)$ 
is $L^I(B)/uL^I(B)$ regular. So $v$ is $U$ regular. So we have 
\[0 \longrightarrow U(-1) \longrightarrow U \longrightarrow \frac{U}{vU} \longrightarrow 0\]
an exact sequence. By Hilbert series $(U/vU)_j= 0$ for $j \geq 2.$ Now by  setting 
$K = \ker \left(( \widetilde{L}^J(B)/u\widetilde{L}^J(B))(-1) \stackrel{v} \longrightarrow  \widetilde{L}^J(B)/u\widetilde{L}^J(B) \right)$
we get by Snake Lemma
 $K_j=0$ for $j\geq 2.$ Also  note that $K_0 = 0.$

\textit{Claim}: $K_1 = 0.$ To prove the claim set $\mathcal{F}= \{ \widetilde{J^n}\}.$ Then $\mathcal{F}$ is a filtration on $B.$ 
Then $\overline{\mathcal{F}} = \{ \widetilde{J^n}+ (x)/(x)\}$ is the quotient filtration on $B/xB.$ Put $\mathfrak{q}=J/(x) = \bar{\mathcal{F}_1}.$
We may assume that $\mathfrak{q}$ is integrally closed. As $\widetilde{J^n} = J^n$ for $n \gg 0.$ So we get $\bar{\mathcal{F}_n}= \mathfrak{q}^n$ for $n \gg 0.$

We prove that $\overline{\mathcal{F}_2}:y = \bar{\mathcal{F}_1} = \mathfrak{q}.$ Let $a \in \overline{\mathcal{F}_2}:y.$ So $ya \in \overline{\mathcal{F}_2}.$
So $y^{n+1}a \in y^n\overline{\mathcal{F}_2} \subset \overline{\mathcal{F}_{2+n}} = \mathfrak{q}^{n+2}$ for $n \gg 0.$ 
This implies $a \in \widetilde{q}=\bar{ \mathfrak{q}}= \mathfrak{q}.$ It follows that $K_1 =0$. Thus $K$ is zero. So $\widetilde{L}^J(B)$
has $\depth \geq 2$. This implies $\widetilde{Gr}_J(B)$ is Cohen-Macaulay. So $Gr_{J^n}(B) $ is Cohen-Macaulay for $n \gg 0$.
\end{proof}
Here we give an example where our Theorem \ref{extn-sally-thrm} holds:
\begin{example}
 Let $A=\mathbb{Q}[|X,Y,Z,W|]/(X^2-Y^2Z, XY^4-Z^2).$ Let $x,y,z,w$ denotes  the images of $X,Y,Z,W$ in $A$ respectively. Let $I=(x,y,w)$ and
 $\m=(x,y,z,w)$. Then $(A, \m) $ is a two dimensional Cohen-Macaulay local ring. Using CoCoA (see, \cite{cocoa}) we have computed 
 \[P_I(t) = \frac{2+2t}{(1-t)^2} \quad \text{and} \quad P_{\m}(t)=\frac{1+2t+t^2}{(1-t)^2}.\]
 We have $e_1^{\m}(A)=e_1^I(A)+\lambda(\m/I)+1.$ We have also checked that $\lambda(\widetilde{\m^2}/I^2)=3=2\lambda(\m/I)+1.$ Hence by 
 Theorem \ref{extn-sally-thrm}$(b)$ we get $\widetilde{Gr}_{\m}(A)$ is Cohen-Macaulay.
\end{example}

\section{The case of third Hilbert coefficient}
In this section we deal with third Hilbert coefficients and generalize a 
remarkable result of Itoh for normal ideals(see, \cite{itoh1}).

\begin{theorem}\label{e_3-an}
 Let $A=B$ and $\psi =id_A$.  Let $\dim A \geq 3$. Let  $I,J$ be   $\mathfrak{m}-$primary ideals and $I$ is a reduction of $J$.
Suppose  $Gr_I(A)$ is Cohen-Macaulay and $J$ is asymptotically normal. Then  \[ e_3^J(A) \geq e_3^I(A).\]
\end{theorem}
\begin{proof}
By standard argument  it suffices to consider $\dim A =3$.  As $J$ is asymptotically normal by \cite[Theorem 3.1]{HH} there exists $n_0$ such
 that   $\depth Gr_{J^n}(A) \geq 2$ for all $n \geq n_0$.
 
 Now set $T=I^n$ , $K=J^n$ and  $W =\bigoplus _{n \geq 0}K^{n+1}/T^{n+1}$ for $n \geq n_0$.  Then we get an exact sequence 
 \[0\longrightarrow W \longrightarrow L^T(A) \longrightarrow L^K(A) \longrightarrow 0\]
  of $\mathcal{R}(T, A)-$modules. Note that  $\depth L^K(A) \geq 2$. So we get $W$ is Cohen-Macaulay of dimension $3$. Thus $e_i(W) \geq 0$ for $0 \leq i \leq 3$. Hence
 \begin{align*}
  e_3^K(A) & = e_3^T(A) + e_2(W).\\
  & \geq e_3^T(A) .
   \end{align*}
 As $e_3^J(A)= e_3^{J^n}(A)$ and $e_3^I(A)=e_3^{I^n}(A)$ for all $n \geq 1$. Therefore  $ e_3^J(A) \geq e_3^I(A) $.
   \end{proof}
   By analysing the case of equality we prove the following:
  \begin{theorem}\label{e_3-euality}
   Let $A=B$ and $\psi =id_A$. Assume  $\dim A = 3$. Let  $I,J$ be two  $\mathfrak{m}-$primary ideals and $I$ is a reduction of $J$.
Suppose  $Gr_I(A)$ is Cohen-Macaulay and $J$ is asymptotically normal. If $e_3^J(A)=e_3^I(A)$ then $Gr_{J^n}(A)$ 
is Cohen-Macaulay for all $n \gg 0$.
  \end{theorem}
 \begin{proof}
 As $J$ is asymptotically normal by \cite[Theorem 3.1]{HH} there exists $n_0$ such
 that for all $n \geq n_0$,  $\depth Gr_{J^n}(A) \geq 2$. 
  Now set $T=I^n$ , $K=J^n$ and  $W =\bigoplus _{n \geq 0}K^{n+1}/T^{n+1}$ for $n \geq n_0$. From 
  the proof of Theorem \ref{e_3-an} we see
 that $W$ is Cohen-Macaulay of dimension $3$ and  $e_3^J(A) = e_3^I(A) + e_2(W).$ 
 
 Suppose $e_3^I(A)=e_3^J(A)$. Then $e_2(W)=0$. So the Hilbert series of $W$ is given by
 \[H_W(s) = \frac{r_0+r_1s}{(1-s)^3}.\]
 
 Let $x,y,z$ be a $K\oplus T$ superficial sequence. Set $u=xt, v=yt$ and $w=zt$. Note that $u,v,w \in \mathcal{R}(T, A)_1.$  Also set 
 $U= W/(u,v)W, \ \overline{L^K(A)}= L^K(A)/(u,v)L^K(A) $ and $\overline{L^T(A)}=L^T(A)/(u,v)L^T(A)$. Then we get an exact sequence
 \[0\longrightarrow U \longrightarrow \overline{L^T(A)} \longrightarrow \overline{L^K(A)}\longrightarrow 0 .\]
 Now consider the commutative diagram 
  \begin{center}
$
\begin{CD}
0 @>>> U(-1) @>>> \overline{L^T(A)}(-1) @>>> \overline{L^K(A)}(-1) @>>> 0 \\
  @.       @VV\text{w}V @VV\text{w}V @VV\text{w}V\\
0 @>>> U @>>> \overline{L^T(A)} @>>> \overline{L^K(A)} @>>> 0 
\end{CD}
$
\end{center}
 
 Also note that Hilbert series of $U/wU$ is given by
 \[H_{U/wU}(s) = r_0 + r_1s.\]
 Therefore $(U/wU)_j=0$ for $j\geq 2.$ Now set $E= ker\left(\overline{L^K(A)}(-1) \stackrel{w}\longrightarrow \overline{L^K(A)}\right)$. Note that
 $E_j=0 $ for all $j \geq 2$  by Snake Lemma. Also  $E_0=0$.\\
 $\textit{Claim}$ $E_1=0$: To prove this set $\mathcal{F}= \{ K^m \}.$ Then $\mathcal{F}$ is a filtration on $A.$ 
Also  $\overline{\mathcal{F}} = \{ K^m+ (x, y)/(x,y)\}$ is the quotient filtration on $A/(x,y)A.$ Put $\mathfrak{q}=K/(x,y) = \bar{\mathcal{F}_1}.$
We may assume that $\mathfrak{q}$ is integrally closed.  Note $\bar{\mathcal{F}_m}= \mathfrak{q}^m$ for $m \geq 1.$

We prove that $\overline{\mathcal{F}_2}:z = \bar{\mathcal{F}_1} = \mathfrak{q}.$ Let $a \in \overline{\mathcal{F}_2}:z.$ So $za \in \overline{\mathcal{F}_2} =\mathfrak{q}^2.$
This implies $a \in \widetilde{q}\subset \bar{ \mathfrak{q}}= \mathfrak{q}.$ It follows that $E_1 =0$. Thus $E$ is zero. So $\depth \overline{L^K(A)}\geq 1$.
Thus $\depth L^K(A) \geq 3.$ Therefore $\depth Gr_K(A) \geq 3.$ Hence $Gr_K(A)$ is Cohen-Macaulay.
\end{proof}

\section{Examples.}
In this section we show that  there are plenty of examples where  Theorem \ref{northcott-extn} holds.
\begin{example}
 Let $(R, \mathfrak{m})$ be a regular local ring. Let $(B, \mathfrak{n}) = (R/\mathfrak{a}, \mathfrak{m}/\mathfrak{a})$
 be a Cohen-Macaulay local ring. Suppose $\dim R = t$ and $\dim B = d.$ Then ht$(\mathfrak{a})= t-d.$ Set $g = t-d.$ Then there exists a regular 
 sequence $\underline{u}= u_1, \cdots , u_g$ of length $g.$ Set $A = R/(\underline{u}).$ Then we get a surjective ring homomorphism
 \[A \stackrel{\psi}\twoheadrightarrow B.\]
 Let $\mathfrak{q}$ be a minimal reduction of $\mathfrak{m}_A.$ Set $I = (\mathfrak{q} :_A \mathfrak{m}_A).$ 
 Clearly $\mathfrak{q} \subset I
  \subset \mathfrak{m}_A. $ By \cite[Theorem 2.1]{cpv} we get
  \[I^2 = \mathfrak{q}I.\]
  So $\psi(I^2)= \psi(\mathfrak{q})\psi(I)$. Thus  $Gr_I(B)$ has minimal multiplicity. As $\psi(\mathfrak{q})B$ is a minimal reduction 
  $\psi(\mathfrak{m}_A)= \mathfrak{n}$ so we get $\psi(I)B$ is a reduction of $\psi(\mathfrak{m}_A)=\mathfrak{n}.$ Hence by 
  Theorem $\ref{northcott-extn}$ we get
  \[e^{\mathfrak{n}}_1(B) \geq e^I_1(B) + \lambda(\mathfrak{n}/I).\]
\end{example}

\begin{example}
 Suppose $(A, \mathfrak{m})$ be a Gorenstein local ring which not regular. Let $J$ be any $\mathfrak{m}-$primary ideal.
 Set $I := (\mathfrak{q} :_A \mathfrak{m})$ where $\mathfrak{q}$ is a minimal reduction of $J$. Then
 \[e^J_1(A) \geq e^I_1(A) + \lambda(J/I).\]
\end{example}
\begin{proof} 
 It is enough to prove that $I$ has minimal multiplicity.   By \cite[Theorem 2.1]{cpv} we have $\mathfrak{q} \subset I \subset J$ and
 $I^2=\mathfrak{q}I.$ Thus 
 $I$ has minimal multiplicity.
\end{proof}
\begin{example}
Let $(A , \mathfrak{m}) \stackrel{\psi} \rightarrow (B , \mathfrak{n})$ be a local homomorphism of Cohen-Macaulay local
 rings with $\dim A = \dim B.$
 Let $I$ be an $\mathfrak{m}$ primary ideal in $A.$ If $A$ is regular and $Gr_I(A)$ is Cohen-Macaulay. Then $Gr_I(B)$ is 
 Cohen-Macaulay. If $\psi(I)B$ is a reduction $J$ in $B$ then
 \[e^J_1(B) \geq e^I_1(B) + \lambda(J/I).\]
 \end{example}
\begin{proof}
 By Auslander-Buchsbaum formula we get $B$ is free as an $A-$module. As $Gr_I(A)$ Cohen-Macaulay, so $Gr_I(B)$ is Cohen-Macaulay. Hence by
 Theorem \ref{northcott-extn} we get the inequality.
\end{proof}
\begin{example}
Let $(A, \mathfrak{m})$ be Cohen-Macaulay local ring and $I$ be an $\mathfrak{m}-$primary ideal. Let $I\subset J \subset \bar{I}$
(integral closure of $I$). If $Gr_I(A)$ is Cohen-Macalay then by Theorem $\ref{northcott-extn}$ we get
\[e^J_1(A) \geq e^I_1(A) + \lambda(J/I).\]
If equality holds above then $Gr_J(A)$ is Cohen-Macaulay.
\end{example}

\section*{Acknowledgements}
The first and last authors would like to thank IIT Bombay especially Department of Mathematics for it's hospitality during the preparation of this work.

\providecommand{\bysame}{\leavevmode\hbox
to3em{\hrulefill}\thinspace}
\providecommand{\MRhref}[2]{%
  \href{http://www.ams.org/mathscinet-getitem?mr=#1}{#2}
} \providecommand{\href}[2]{#2}

\end{document}